\documentclass{llncs}
\usepackage{epsfig}
\usepackage{graphicx}
\usepackage{amsfonts}
\usepackage{amssymb}
\usepackage[linesnumbered,boxed,commentsnumbered]{algorithm2e}

\def\bsquareforqed{\rule{0.6em}{0.6em}}
\def\bqed{\ifmmode\bsquareforqed\else{\unskip\nobreak\hfil
\penalty50\hskip1em\null\nobreak\hfil\bsquareforqed
\parfillskip=0pt\finalhyphendemerits=0\endgraf}\fi}
\spnewtheorem{observation}{Observation}{\bfseries}{\itshape}
\spnewtheorem{clm}{Claim}{\bfseries}{\rmfamily}
\newcommand{\rlzr}{\mathcal{R}}
\newcommand{\comp}[1]{\overline{#1}}  

\newcommand{\boxi}{\mbox{\textnormal{box}}}

\newcommand{\prob}{\mbox{\textnormal{Prob}}}

\newcommand{\tw}{\hspace{0.5mm}\mbox{\textnormal{tree-width}\hspace{.5mm}}}
\newcommand{\mvc}{\hspace{0.5mm}\mbox{\textnormal{MVC}\hspace{.5mm}}}
\newcommand{\ceil}[1]{\left\lceil #1 \right\rceil}
\newcommand{\floor}[1]{\left\lfloor #1 \right\rfloor}

\newcommand{\poset}{\mathcal{P}}

\newcommand{\igr}{\mathcal{I}}

\begin{document}
\title{Boxicity and Poset Dimension}
\author{Abhijin Adiga\inst{1}, Diptendu Bhowmick\inst{1}, L. Sunil
Chandran\inst{1}}
\institute{Department of Computer Science and Automation, Indian Institute
of Science, Bangalore--560012, India. \\emails: abhijin@csa.iisc.ernet.in, diptendubhowmick@gmail.com, sunil@csa.iisc.ernet.in}
\date{}
\maketitle
{\renewcommand{\thefootnote}{}
\footnote{This work was supported by DST grant SR/S3/EECE/62/2006
and Infosys Fellowship.}}
\begin{abstract}
Let $G$ be a simple, undirected, finite graph with vertex set $V(G)$ and
edge set $E(G)$. A $k$-dimensional
box is a Cartesian product of closed intervals $[a_1,b_1]\times
[a_2,b_2]\times\cdots\times [a_k,b_k]$.  The {\it boxicity} of $G$,
$\boxi(G)$ is the minimum integer $k$ such that $G$ can be represented
as the intersection graph of $k$-dimensional boxes, i.e. each vertex
is mapped to a $k$-dimensional box and two vertices are adjacent in $G$
if and only if their corresponding boxes intersect. Let $\poset=(S,P)$
be a poset where $S$ is the ground set and $P$ is a reflexive,
anti-symmetric and transitive binary relation on $S$. The dimension
of $\poset$, $\dim(\poset)$ is the minimum integer $t$ such that $P$
can be expressed as the intersection of $t$ total orders.

Let $G_\poset$ be the \emph{underlying comparability graph} of $\poset$,
i.e. $S$ is the vertex set and two vertices are adjacent if and only if
they are comparable in $\poset$. It is a well-known fact that posets
with the same underlying comparability graph have the same dimension.
The first result of this paper links the dimension of a poset to the
boxicity of its underlying comparability graph. In particular, we
show that for any poset $\poset$, $\boxi(G_\poset)/(\chi(G_\poset)-1)
\le \dim(\poset)\le 2\boxi(G_\poset)$, where $\chi(G_\poset)$ is the
chromatic number of $G_\poset$ and $\chi(G_\poset)\ne1$. It immediately
follows that if $\poset$ is a height-$2$ poset, then $\boxi(G_\poset)\le
\dim(\poset)\le 2\boxi(G_\poset)$ since the underlying comparability
graph of a height-$2$ poset is a bipartite graph.

The second result of the paper relates the boxicity of a graph $G$ with a
natural partial order associated with the \emph{extended double cover} of
$G$, denoted as $G_c$: Note that $G_c$ is a bipartite graph with partite
sets $A$ and $B$ which are copies of $V(G)$ such that corresponding
to every $u\in V(G)$, there are two vertices $u_A\in A$ and $u_B\in B$
and $\{u_A,v_B\}$ is an edge in $G_c$ if and only if either $u=v$ or $u$
is adjacent to $v$ in $G$. Let $\poset_c$ be the natural height-$2$ poset
associated with $G_c$ by making $A$ the set of minimal elements and
$B$ the set of maximal elements. We show that $\frac{\boxi(G)}{2}
\le \dim(\poset_c) \le 2\boxi(G)+4$.

These results have some immediate and significant consequences.
The upper bound $\dim(\poset)\le 2\boxi(G_\poset)$ allows us to derive
hitherto unknown upper bounds for poset dimension such as
$\dim(\poset)\le 2\tw(G_\poset)+4$, since boxicity of any graph is known to
be at most its $\tw+2$. In the other direction, using the
already known bounds for partial order dimension
we get the following: (1) The boxicity of any graph
with maximum degree $\Delta$ is $O(\Delta\log^2\Delta)$ which is an
improvement over the best known upper bound of $\Delta^2+2$. (2)
There exist graphs with boxicity $\Omega(\Delta\log\Delta)$. This
disproves a conjecture that the boxicity of a graph is $O(\Delta)$. (3)
There exists no polynomial-time algorithm to approximate the boxicity of
a bipartite graph on $n$ vertices with a factor of $O(n^{0.5-\epsilon})$
for any $\epsilon>0$, unless $NP=ZPP$.
\end{abstract}

\section{Introduction}\label{sec:intro}
\subsection{Boxicity}
A $k$-box is a Cartesian product of closed intervals $[a_1,b_1]\times
[a_2,b_2]\times\cdots\times [a_k,b_k]$.  A $k$-box representation
of a graph $G$ is a mapping of the vertices of $G$ to $k$-boxes in
the $k$-dimensional Euclidean space such that two vertices in $G$ are
adjacent if and only if their corresponding $k$-boxes have a non-empty
intersection. The \emph{boxicity} of a graph denoted $\boxi(G)$, is the
minimum integer $k$ such that $G$ has a $k$-box representation.  Boxicity
was introduced by Roberts \cite{recentProgressesInCombRoberts}. Cozzens
\cite{phdThesisCozzens} showed that computing the boxicity of
a graph is NP-hard. This was later strengthened by Yannakakis
\cite{complexityPartialOrderDimnYannakakis} and finally by Kratochv\`{\i}l
\cite{specialPlanarSatisfiabilityProbNPKratochvil} who showed that
determining whether boxicity of a graph is at most two itself is
NP-complete.

It is easy to see that a graph has boxicity at most 1 if and only
if it is an \emph{interval graph}, i.e. each vertex of the graph
can be associated with a closed interval on the real line such that
two intervals intersect if and only if the corresponding vertices are
adjacent. By definition, boxicity of a complete graph is $0$.  Let $G$
be any graph and $G_i$, $1\le i\le k$ be graphs on the same vertex set
as $G$ such that $E(G)=E(G_1)\cap E(G_2)\cap\cdots\cap E(G_k)$. Then we
say that $G$ is the \emph{intersection graph} of $G_i$ s for $1\le i\le
k$ and denote it as $G=\bigcap_{i=1}^{k}G_i$. Boxicity can be stated in
terms of intersection of interval graphs as follows:
\begin{lemma}{\textnormal{Roberts
\cite{recentProgressesInCombRoberts}:}}\label{lem:intbox} 
The boxicity of a non-complete graph $G$ is the minimum positive integer
$b$ such that $G$ can be represented as the intersection of $b$ interval
graphs. Moreover, if $G=\bigcap_{i=1}^{m}G_i$ for some graphs $G_i$
then $\boxi(G)\leq\sum_{i=1}^{m}\boxi(G_i)$.
\end{lemma}

Roberts, in
his seminal work \cite{recentProgressesInCombRoberts} proved
that the boxicity of a complete $k$-partite graph is $k$.
Chandran and Sivadasan \cite{boxicityTreewidthSunilNaveen}
showed that $\boxi(G)\le\tw(G)+2$. Chandran, Francis and
Sivadasan \cite{boxicityMaxDegreeSunilNaveenFrancis} proved that
$\boxi(G)\le\chi(G^2)$ where, $\chi(G^2)$ is the chromatic number of
$G^2$. In \cite{boxicityGraphsBoundedDegreeEsperet} Esperet proved
that $\boxi(G)\le \Delta^2(G)+2$, where $\Delta(G)$ is the maximum
degree of $G$. Scheinerman \cite{phdThesisScheinerman} showed
that the boxicity of outer planar graphs is at most 2. Thomassen
\cite{intervalRepPlanarGraphsThomassen} proved that the boxicity of
planar graphs is at most 3. In \cite{computingBoxicityCozzensRoberts},
Cozzens and Roberts studied the boxicity of split graphs. 

\subsection{Poset Dimension}
A \emph{partially ordered set} or \emph{poset} $\poset=(S,P)$
consists of a non empty set $S$, called the \emph{ground set}
and a reflexive, antisymmetric and transitive binary relation $P$
on $S$. A \emph{total order} is a partial order in which every two
elements are comparable. It essentially corresponds to a permutation
of elements of $S$. A \emph{height-$2$ poset} is one in which
every element is either a minimal element or a maximal element. A
\emph{linear extension} $L$ of a partial order $P$ is a total order
which satisfies $(x\le y \textrm{ in } P\Rightarrow x\le y \textrm{ in
} L)$.  A \emph{realizer} of a poset $\poset=(S,P)$ is a set of linear
extensions of $P$, say $\rlzr$ which satisfy the following condition:
for any two distinct elements $x$ and $y$, $x<y$ in $P$ if and only
if $x<y$ in $L$, $\forall L\in\rlzr$.  The \emph{poset dimension}
of $\poset$ (sometimes abbreviated as dimension of $\poset$) denoted
by $\dim(\poset)$ is the minimum integer $k$ such that there exists
a realizer of $P$ of cardinality $k$. Poset dimension was introduced
by Dushnik and Miller \cite{partiallyOrderedSetsDushnik}. Clearly,
a poset is one-dimensional if and only if it is a total
order. Pnueli et al.  \cite{transOrientGraphsPnueliEtal} gave
a polynomial time algorithm to recognize dimension 2 posets. In
\cite{complexityPartialOrderDimnYannakakis} Yannakakis showed that it
is NP-complete to decide whether the dimension of a poset is at most
3. For more references and survey on dimension theory of posets see
Trotter \cite{combPosetsDimensionTrotter,partiallyOrderedSetsTrotter}.
Recently, Hegde and Jain \cite{hardnessPODHegdeJain} showed that it is
hard to design an approximation algorithm for computing the dimension
of a poset.

A simple undirected graph $G$ is a comparability graph if and only
if there exists some poset $\poset=(S,P)$, such that $S$ is the
vertex set of $G$ and two vertices are adjacent in $G$ if and only
if they are comparable in $\poset$. We will call such a poset an
\emph{associated poset} of $G$. Likewise, we refer to $G$ as the
\emph{underlying comparability graph} of $\poset$. Note that for a
height-$2$ poset, the underlying comparability graph is a bipartite
graph with partite sets $A$ and $B$, with say $A$ corresponding to
minimal elements and $B$ to maximal elements. For more on comparability
graphs see \cite{algGraphTheoryPerfectGraphsGolumbic}. It is
easy to see that there is a unique comparability graph associated
with a poset, whereas, there can be several posets with the same
underlying comparability graph. However, Trotter, Moore and Sumner
\cite{dimnComparabilityGraphTrotterMooreSumner} proved that posets with
the same underlying comparability graph have the same dimension.

\section{Our Main Results}\label{sec:mainResults}
The results of this paper are the consequence of our attempts to bring out
some connections between boxicity and poset dimension. As early as 1982,
Yannakakis had some intuition regarding a possible connection between
these problems when he established the NP-completeness of both poset
dimension and boxicity in \cite{complexityPartialOrderDimnYannakakis}. But interestingly, no results were discovered in
the last 25 years which establish links between these
two notions. Perhaps the researchers were misled by some deceptive examples
such as the following one: Consider a complete graph $K_n$ where $n$ is
even and remove a perfect matching from it. The resulting graph is a
comparability graph and the dimension of any of its associated posets is 2,
while its boxicity is $n/2$. 
In this context it may be worth recalling a result from
\cite{geometricContainmentOrdersFishburnTrotter} which relates the poset
dimension to another parameter namely the dimension of box orders. A poset
$\poset=(S,P)$ is said to be a box order in $m$ dimensions if there exists
a mapping of its elements to $m$-dimensional axis-parallel boxes such
that $x<y$ in $P$ if and only if the box of $y$ strictly contains the
box of $x$. Box order is a particular type of geometrical containment
order (See \cite{geometricContainmentOrdersFishburnTrotter,combPosetsDimensionTrotter}). The
result is as follows: the dimension of $\poset$ is at
most $2m$ if and only if it is a box order in $m$ dimensions
\cite{containmentIntersectionGolumbic,containmentGraphsPosetsGolumbicScheinerman}.
But note that boxicity is fundamentally different from box orders. As in
the case of the above example,
we can demonstrate families of posets of
constant dimension whose underlying comparability graphs have arbitrarily
high boxicity, which is in contrast with the above result on box orders.


First we state an upper bound and a lower bound for the dimension of a
poset in terms of the boxicity of its underlying comparability graph.
\begin{theorem}\label{thm:dimUpperBound}
Let $\poset=(V,P)$ be a poset such that $\dim(\poset)>1$ and $G_\poset$
its underlying comparability graph. Then,
$\dim(\poset) \le 2\boxi(G_\poset).$
\end{theorem}
\begin{theorem}\label{thm:lowerBoundPosetdim}
Let $\poset=(V,P)$ be a poset and let $\chi(G_\poset)$ be the
chromatic number of its underlying comparability graph $G_\poset$
such that $\chi(G_\poset)>1$. Then,
$\dim(\poset)\ge\frac{\boxi(G_\poset)}{\chi(G_\poset)-1}$.
\end{theorem}
Note that if $\poset$ is a height-$2$ poset, then $G_\poset$ is a
bipartite graph and therefore $\chi(G_\poset)=2$.  Thus, from the above
results we have the following:
\begin{corollary}\label{cor:height2Posetdim}
Let $\poset=(V,P)$ be a height-$2$ poset and $G_\poset$ its
underlying comparability graph. Then, $\boxi(G_\poset)\le
\dim(\poset)\le2\boxi(G_\poset)$.
\end{corollary}
The \emph{double cover} and \emph{extended double cover} of a graph are
popular notions in graph theory. They provide a natural way to associate a
bipartite graph to the given graph. In this paper we make use of the latter
construction.
\begin{definition}\label{def:coverGraph}
The extended double cover of $G$, denoted as $G_c$ is a bipartite
graph with partite sets $A$ and $B$ which are copies of $V(G)$ such that
corresponding to every $u\in V(G)$, there are two vertices $u_A\in A$ and
$u_B\in B$ and $\{u_A,v_B\}$ is an edge in $G_c$ if and only if either
$u=v$ or $u$ is adjacent to $v$ in $G$.
\end{definition}
We prove the following lemma relating the boxicity of $G$ and $G_c$.
\begin{lemma}\label{lem:boxCovGraph}
Let $G$ be any graph and $G_c$ its extended double cover. Then,
\[
\frac{\boxi(G)}{2} \le \boxi(G_c) \le \boxi(G)+2.
\]
\end{lemma}
Let $\poset_c$ be the natural height-$2$ poset associated with
$G_c$, i.e. the elements in $A$ are the minimal elements and the
elements in $B$ are the maximal elements. Combining Corollary
\ref{cor:height2Posetdim} and Lemma \ref{lem:boxCovGraph} we have the
following theorem:
\begin{theorem}
Let $G$ be a graph and $\poset_c$ be the natural height-$2$
poset associated with its extended double cover. Then,
$\frac{\dim(\poset_c)}{2}-2\le\boxi(G)\le 2\dim(\poset_c)$ and therefore
$\boxi(G)=\Theta(\dim(\poset_c))$.
\end{theorem}

\subsection{Consequences}
\subsubsection{New upper bounds for poset dimension:}
Our results lead to some hitherto
unknown bounds for poset dimension. Some general bounds obtained
in this manner are listed below:
\begin{enumerate}
\item It is proved in \cite{boxicityTreewidthSunilNaveen} that for any
graph $G$, boxicity of $G$ is at most $\tw(G)+2$. For more information on
\emph{tree-width} see \cite{touristGuideTreewidthBodlaender}.  Applying
this bound in Theorem \ref{thm:dimUpperBound} it immediately follows
that, for a poset $\poset$, $\dim(\poset)\le 2\tw(G_\poset)+4$.
\item The \emph{threshold dimension} of a graph $G$ is the minimum
number of \emph{threshold graphs} such that $G$ is the edge union of
these graphs. For more on threshold graphs and threshold
dimension see \cite{algGraphTheoryPerfectGraphsGolumbic}. Cozzens and
Halsey \cite{relationshipTgSplitCozzensHalsey} proved that $\boxi(G)\le
\textrm{threshold-dimension}(\comp{G})$, where $\comp{G}$ is
the complement of $G$. From this it follows that
$\dim(\poset)\le \textrm{$2$ threshold-dimension}(\comp{G_\poset})$.
\item In \cite{cubicityBoxVCSunilAnitaChintan} it is proved that
$\boxi(G)\le\floor{\frac{\textrm{MVC$(G)$}}{2}}+1$, where $\mvc(G)$ is
the cardinality of the \emph{minimum vertex cover} of $G$. Therefore,
we have $\dim(\poset)\le\mvc(G_\poset)+2$.
\end{enumerate}
Some more interesting results can be obtained if we restrict $G_\poset$ to
belong to certain subclasses of graphs. Note that there are several
research papers in the partial order literature which study the dimension
of posets whose underlying comparability graph has some special structure
-- interval order, semi order and crown posets are some examples.
\begin{enumerate}
\setcounter{enumi}{3}
\item Scheinerman \cite{phdThesisScheinerman} proved that the
boxicity of outer planar graphs is at most 2 and later Thomassen
\cite{intervalRepPlanarGraphsThomassen} proved that the boxicity of
planar graphs is at most 3.  Therefore, it follows that $\dim(\poset)\le
4$ if $G_\poset$ is outer planar and $\dim(\poset)\le 6$ if $G_\poset$ is
planar.
\item Bhowmick and Chandran \cite{boxCubATFreeGraphsBhowmickChandran}
proved that boxicity of AT-free graphs is at most $\chi(G_\poset)$. Hence, 
$\dim(\poset)\le 2\chi(G_\poset)$, if $G_\poset$ is AT-free. 
\item If $G_\poset$ is an interval graph, then, we get from Theorem
\ref{thm:dimUpperBound}, $\dim(\poset)\le 2$, since $\boxi(G_\poset)=1$.
However, observing that interval graphs are co-comparability graphs
this result would follow also as a consequence of a result by Dushnik
and Miller \cite{partiallyOrderedSetsDushnik}: $\dim(\poset)\le 2$
if and only if $G_\poset$ is a co-comparability graph.
\item The boxicity of a $d$-\emph{dimensional hypercube} is $O(d/\log(d))$
\cite{cubicityHypercubeSunilNaveen}. Therefore, if $G_\poset$
is a height-2 poset which corresponds to a $d$-\emph{dimensional
hypercube}, then from Corollary \ref{cor:height2Posetdim} we have
$\dim(\poset)=O(d/\log(d))$.
\item Chandran et al. \cite{boxicityChordalBipartiteSunilFrancisRogers}
recently proved that chordal bipartite graphs have arbitrarily high
boxicity. From Corollary \ref{cor:height2Posetdim} it follows that
height-$2$ posets whose underlying comparability graph are chordal
bipartite graphs can have arbitrarily high dimension.
\end{enumerate}

\subsubsection{Improved upper bound for boxicity based on maximum degree:}
F\"{u}redi and Kahn \cite{dimensionsOfOrderedSetsFurediKahn}
proved the following
\begin{lemma}\label{thm:furediKahn}
Let $\poset$ be a poset and $\Delta$ be the maximum degree
of $G_\poset$. Then, there exists a constant $c$ such that
$\dim(\poset)<c\Delta(\log\Delta)^2$.
\end{lemma}
From Lemma \ref{lem:boxCovGraph} and Corollary \ref{cor:height2Posetdim}
we have $\boxi(G)\le 2\boxi(G_c)\le 2\dim(\poset_c)$, where $G_c$
is the extended double cover of $G$. Note that by construction
$\Delta(G_c)=\Delta(G)+1$. On applying the above lemma, we have
\begin{theorem}
For any graph $G$ having maximum degree $\Delta$ there exists a constant
$c'$ such that $\boxi(G) <c'\Delta(\log\Delta)^2$. 
\end{theorem}
This result is an improvement over the previous upper bound of $\Delta^2+2$
by Esperet \cite{boxicityGraphsBoundedDegreeEsperet}.

\subsubsection{Counter examples to the conjecture of
\cite{boxicityMaxDegreeSunilNaveenFrancis}:} Chandran et
al. \cite{boxicityMaxDegreeSunilNaveenFrancis} conjectured that boxicity
of a graph is $O(\Delta)$. We use a result by Erd\H{o}s, Kierstead and
Trotter \cite{dimensionRandomOrderedSetsErdosEtal} to show that there
exist graphs with boxicity $\Omega(\Delta\log\Delta)$, hence disproving
the conjecture.  Let $\mathbb{P}(n,p)$ be the probability space of
height-$2$ posets with $n$ minimal elements forming set $A$ and $n$
maximal elements forming set $B$, where for any $a\in A$ and $b\in B$,
$\prob(a<b)=p(n)=p$.  They proved the following:
\begin{theorem}{\textnormal{\cite{dimensionRandomOrderedSetsErdosEtal}}}
For every $\epsilon>0$, there exists $\delta>0$ so that if
$(\log^{1+\epsilon}n)/n<p<1-n^{-1+\epsilon}$, then, $\dim(\poset) > (\delta
pn\log(pn))/(1+\delta p\log(pn))$ for almost all $\poset\in \mathbb{P}(n,p)$.
\end{theorem}
When $p=1/\log n$, for almost all posets $\poset\in \mathbb{P}(n,1/\log
n)$, $\Delta(G_\poset)<\delta_1n/\log n$ and by the above theorem
$\dim(\poset)>\delta_2n$, where $\delta_1$ and $\delta_2$ are some
positive constants (see \cite{partiallyOrderedSetsTrotter} for a
discussion on the above theorem). From Theorem \ref{thm:dimUpperBound},
it immediately implies that for almost all $\poset\in
\mathbb{P}(n,1/\log n)$, $\boxi(G_\poset)\ge\frac{\dim(\poset)}{2}
> \delta'\Delta(G_\poset)\log \Delta(G_\poset)$ for some positive
constant $\delta'$, hence proving the existence of graphs with boxicity
$\Omega(\Delta\log\Delta)$.

\subsubsection{Approximation hardness for the boxicity of bipartite graphs:}
Hegde and Jain \cite{hardnessPODHegdeJain} proved the following
\begin{theorem}\label{thm:hegdeJain}
There exists no polynomial-time algorithm to approximate the 
dimension of an $n$-element poset within a factor of $O(n^{0.5-\epsilon})$
for any $\epsilon >0$, unless $NP=ZPP$.
\end{theorem}
This is achieved by reducing the \emph{fractional chromatic number problem}
on graphs to the poset dimension problem. In addition they observed that a
slight modification of their reduction will imply the same result for even
height-$2$ posets. From Corollary \ref{cor:height2Posetdim}, it is clear that for
any height-$2$ poset $\poset$, $\dim(\poset)=\Theta(\boxi(G_\poset))$.
Suppose there exists an algorithm to compute the boxicity of bipartite
graphs with approximation factor $O(n^{0.5-\epsilon})$, for some
$\epsilon>0$, then, it is clear that the same algorithm can be used to
compute the dimension of height-$2$ posets with approximation factor
$O(n^{0.5-\epsilon})$, a contradiction. Hence,
\begin{theorem}
There exists no polynomial-time algorithm to approximate the boxicity
of a bipartite graph on $n$-vertices with a factor of $O(n^{0.5-\epsilon})$
for any $\epsilon >0$, unless $NP=ZPP$.
\end{theorem}

\section{Notations}
Let $[n]$ denote $\{1,2,\ldots,n\}$ where $n$ is a positive integer. For
any graph $G$, let $V(G)$ and $E(G)$ denote its vertex set and edge set
respectively. If $G$ is undirected, for any $u,v\in V(G)$, $\{u,v\}\in
E(G)$ means $u$ is adjacent to $v$ and if $G$ is directed, $(u,v)\in E(G)$
means there is a directed edge from $u$ to $v$. Whenever we refer to an
$AB$ bipartite (or co-bipartite) graph, we imply that its vertex set is
partitioned into non-empty sets $A$ and $B$ where both these sets induce
independent sets (cliques respectively).

In a poset $\poset=(S,P)$, the notations $aPb$, $a\le b$ in $P$ and
$(a,b)\in P$ are equivalent and are used interchangeably. $G_\poset$
denotes the underlying comparability graph of $\poset$. A subset of
$\poset$ is called a \emph{chain} if each pair of distinct elements is
comparable. If each pair of distinct elements is incomparable, then
it is called an \emph{antichain}. Given an $AB$ bipartite graph $G$,
the natural poset associated with $G$ with respect to the bipartition
is the poset obtained by taking $A$ to be the set of minimal elements
and $B$ to be the set of maximal elements. In particular, if $G_c$
is the extended double cover of $G$, we denote by $\poset_c$ the natural
associated poset of $G_c$.

Suppose $I$ is an interval graph. Let $f_I$ be an \emph{interval
representation} for $I$, i.e. it is a mapping from the vertex
set to closed intervals on the real line such that for any two
vertices $u$ and $v$, $\{u,v\}\in E(I)$ if and only if $f_I(u)\cap
f_I(v)\ne\varnothing$. Let $l(u,f_I)$ and $r(u,f_I)$ denote the left
and right end points of the interval corresponding to the vertex $u$
respectively. In this paper, we will never consider more than one interval
representation for an interval graph. Therefore, we will simplify the
notations to $l(u,I)$ and $r(u,I)$. Further, when there is no ambiguity
about the graph under consideration and its interval representation,
we simply denote the left and right end points as $l(u)$ and $r(u)$
respectively. Note that for any interval graph there exists an interval
representation with all end points distinct. Such a representation is
called a \emph{distinguishing} interval representation. It is an easy
exercise to derive such a distinguishing interval representation starting
from an arbitrary interval representation of the graph.

\section{Proof of Theorem \ref{thm:dimUpperBound}}
Let $\boxi(G_\poset)=k$. Note that since $\dim(\poset)>1$, $G_\poset$
cannot be a complete graph and therefore $k\ge1$. Let
$\igr=\{I_1,I_2,\ldots,I_k\}$ be a set of interval graphs such
that $G_\poset=\bigcap_{i=1}^kI_i$. Now, corresponding to each $I_i$
we will construct two total orders $L_i^1$ and $L_i^2$ such that
$\rlzr=\{L_i^j|i\in[k]\textrm{ and }j\in[2]\}$ is a realizer of $\poset$.

\newcommand{\oP}{\overline{P}}
Let $I\in\igr$ and $f_I$ be an interval representation of $I$. We will
define two partial orders $P_I$ and $\oP_I$ as follows: $\forall a\in V$,
$(a,a)$ belongs to $P_I$ and $\oP_I$ and for every
non-adjacent pair of vertices $a,b\in V$ with respect to $I$,
\[
\left.
\begin{array}{l}
(a,b)\in P_I \\
(b,a)\in \oP_I
\end{array}\right\} \textrm{ if and only if } r(a,f_I)<l(b,f_I).
\]
Partial orders constructed in the above manner from a collection of closed
intervals are called \emph{interval orders} (See
\cite{partiallyOrderedSetsTrotter} for more details).
It is easy to see that $\comp{I}$ (the complement of $I$) is the underlying
comparability graph of both $P_I$ and $\oP_I$.

Let $G_1$ and $G_2$ be two directed graphs with vertex set $V$
and edge set $E(G_1)=(P\cup P_I)\setminus\{(a,a)|a\in V\}$ and
$E(G_2)=(P\cup\oP_I)\setminus\{(a,a)|a\in V\}$ respectively. Note that
from the definition it is obvious that there are no directed loops in
$G_1$ and $G_2$.
\begin{lemma}
$G_1$ and $G_2$ are acyclic directed graphs.
\end{lemma}
\begin{proof}
We will prove the lemma for $G_1$ -- a similar proof holds for $G_2$.
First of all, since $G_\poset$ is not a complete graph $P_I\ne\varnothing$.
Suppose $P_I$ is a total order, i.e. if $P$ is an antichain, then it is
clear that $E(G_1)=P_I$ and therefore $G_1$ is acyclic. Henceforth, we
will assume that $P_I$ is not a total order.

Suppose $G_1$ is not acyclic. Let
$C=\{(a_0,a_1),(a_1,a_2),\ldots,(a_{t-2},a_{t-1}),(a_{t-1},a_1)\}$
be a shortest directed cycle in $G_1$.

First we will show that $t>2$ ($t$ is the length of $C$). If $t=2$, then
there should be $a,b\in V$ such that $(a,b),(b,a)\in E(G_1)$. Since $P$ is
a partial order, $(a,b)$ and $(b,a)$ cannot be simultaneously present in
$P$. The same holds for $P_I$. Thus, without loss of generality we can
assume that $(a,b)\in P$ and $(b,a)\in P_I$. But if $(a,b)\in P$, then, $a$
and $b$ are adjacent in $G_\poset$ and thus adjacent in $I$. Then clearly
the intervals of $a$ and $b$ intersect and therefore $(b,a)\notin P_I$, a
contradiction.

Now, we claim that two consecutive edges in $C$ cannot belong to $P$
(or $P_I$). Suppose there do exist such edges, say $(a_i,a_{i+1})$
and $(a_{i+1},a_{i+2})$ which belong to $P$ (or $P_I$) (note that the
addition is modulo $t$). Since $P$ (or $P_I$) is a partial order, it
implies that $(a_i,a_{i+2})\in P$ (or $P_I$) and as a result we have a
directed cycle of length $t-1$, a contradiction to the assumption that
$C$ is a shortest directed cycle. Therefore, the edges of $C$ alternate
between $P$ and $P_I$. It also follows that $t\ge 4$.

Without loss of generality we will assume that $(a_1,a_2),(a_3,a_4)\in
P_I$. We claim that $\{(a_1,a_2),(a_3,a_4)\}$ is an induced
poset of $P_I$. First of all $a_2$ and $a_3$ are not comparable
in $P_I$ as they are comparable in $P$. If either $\{a_1,a_3\}$
or $\{a_2,a_4\}$ are comparable, then we can demonstrate a shorter
directed cycle in $G_1$, a contradiction. Finally we consider the
pair $\{a_1,a_4\}$. If $t=4$, then they are not comparable as they are
comparable in $P$ while if $t\ne4$ and if they are comparable, then,
it would again imply a shorter directed cycle, a contradiction. Hence,
$\{(a_1,a_2),(a_3,a_4)\}$ is an induced subposet. In the literature
such a poset is denoted as $\textbf{2}+\textbf{2}$ where $+$ refers
to \emph{disjoint sum} and $\textbf{2}$ is a two-element total
order. Fishburn \cite{intransitiveIndifferenceFishburn} has proved
that a poset is an interval order if and only if it does not contain
a $\textbf{2}+\textbf{2}$. This implies that $P_I$ is not an interval
order, a contradiction.

We have therefore proved that there cannot be any directed cycles in
$G_1$. In a similar way we can show that $G_2$ is an acyclic directed
graph.  \qed
\end{proof}

Since $G_1$ and $G_2$ are acyclic, we can construct total orders, say
$L^1$ and $L^2$ using \emph{topological sort} on $G_1$ and $G_2$ such
that $P\cup P_I\subseteq L^1$ and $P\cup \oP_I\subseteq L^2$ (For more
details on topological sort, see \cite{introAlgoCormenEtal} for example).

For each $I_i$, we create linear extensions $L_i^1$ and $L_i^2$ as
described above. We claim that $\rlzr=\{L_i^j|i\in[k], j\in[2]\}$ is a
realizer of $\poset$. For each $L_i^j$, it is clear from construction
that $P\subset L_i^j$.  If $a$ and $b$ are not comparable in $P$, then
$\{a,b\}\notin E(G_\poset)$, and therefore there exists some interval
graph $I_q\in\igr$ such that $\{a,b\}\notin E(I_q)$. Assuming that
the interval for $a$ occurs before the interval for $b$ in the interval
representation of $I_q$, it follows by construction that $(a,b)\in P_{Iq}$
and $(b,a)\in \oP_{Iq}$ and therefore $(a,b)\in L^1_q$ and $(b,a)\in
L^2_q$. Hence proved.

\subsection{Tight Example for Theorem \ref{thm:dimUpperBound}}
Consider the \emph{crown} poset $S_n^0$: a height-$2$ poset with $n$
minimal elements $a_1,a_2,\ldots,a_{n}$ and $n$ maximal elements
$b_1,b_2,\ldots,b_{n}$ and $a_i<b_j$, for $j=i+1, i+2,\ldots,i-1$,
where the addition is modulo $n$. Its underlying comparability graph is
the bipartite graph obtained by removing a perfect matching from the
complete bipartite graph $K_{n,n}$. The dimension of this poset is $n$
(see \cite{dimensionCrownTrotter,partiallyOrderedSetsTrotter})
while the boxicity of the graph is $\ceil{\frac{n}{2}}$
\cite{cubicityBoxVCSunilAnitaChintan}.

\section{Proof of Theorem \ref{thm:lowerBoundPosetdim}}
We will prove that $\boxi(G_\poset)\le
(\chi(G_\poset)-1)\dim(\poset)$. Let $(\chi(G_\poset)-1)=p$,
$\dim(\poset)=k$ and $\rlzr=\{L_1,\ldots,L_k\}$ a realizer of
$\poset$. Now we color the vertices of $G_\poset$ as follows: For a
vertex $v$, if $\gamma$ is the length of a longest chain in $\poset$
such that $v$ is its maximum element, then we assign color $\gamma$
to it. This is clearly a proper coloring scheme since if two vertices
$x$ and $y$ are assigned the same color, say $\gamma$ and $x<y$, then
it implies that the length of a longest chain in which $y$ occurs
as the maximum element is at least $\gamma+1$, a contradiction. Also,
this is a minimum coloring because the maximum number that gets assigned
to any vertex equals the length of a longest chain in $\poset$, which
corresponds to the clique number of $G_\poset$.

Now we construct $pk$ interval graphs $\igr=\{I_{ij} | i\in[p],
j\in[k]\}$ and show that $G_\poset$ is an intersection graph of
these interval graphs.  Let $\Pi_j$ be the \emph{permutation
induced} by the total order $L_j$ on $[n]$, i.e. $xL_jy$ if
and only if $\Pi_j^{-1}(x)<\Pi_{j}^{-1}(y)$.  The following
construction applies to all graphs in $\igr$ except $I_{pk}$.
Let $I_{ij}\in\igr\setminus\{I_{pk}\}$. We assign the point interval
$\left[\Pi_j^{-1}(v),\Pi_j^{-1}(v)\right]$ for all vertices $v$
colored $i$.  For all vertices $v$ colored $i'<i$, we assign
$\left[\Pi_j^{-1}(v),n+1\right]$ and for those colored $i'>i$, we
assign $\left[0,\Pi_j^{-1}(v)\right]$. The interval assignment for
the last interval graph $I_{pk}$ is as follows: for all vertices
$v$ colored $p+1=\chi(G_\poset)$ we assign the point interval
$\left[\Pi_k^{-1}(v),\Pi_k^{-1}(v)\right]$ and for the rest of the
vertices we assign the interval $\left[\Pi_k^{-1}(v),n+1\right]$. Next,
we will show that $G_\poset=\bigcap_{I\in \igr}I$.
\begin{clm}
If $u$ and $v$ are adjacent in $G_\poset$, then they are adjacent in
all $I\in\igr$.
\end{clm}
\begin{proof}
Let $u$ be colored $i$ and $v$ be colored $i'$. It is clear that $i\ne
i'$ and without loss of generality we will assume that $i<i'$. By the
way we have colored, it implies that $u<v$ in $P$ and therefore
$\Pi_j^{-1}(u)<\Pi_j^{-1}(v)$, $\forall j\in[k]$. Let $I_{hj}$,
$h\in[p]$ and $j\in[k]$ be the interval graph under consideration.
There are 5 possible cases which we consider one by one:
\paragraph{Case 1: $(h<i,i')$} 
By construction in $I_{hj}$, $u$ and $v$ are assigned intervals
$\left[0,\Pi_j^{-1}(u)\right]$ and $\left[0,\Pi_j^{-1}(v)\right]$
respectively and therefore $u$ and $v$ are adjacent in $I_{hj}$, $\forall
j\in [k]$. 
\paragraph{Case 2: $(i,i'<h)$} 
$u$ and $v$ are assigned intervals
$\left[\Pi_j^{-1}(u),n+1\right]$ and $\left[\Pi_j^{-1}(v),n+1\right]$
respectively and therefore are adjacent in $I_{hj}$, $\forall j\in [k]$. 
\paragraph{Case 3: $(i<h<i')$} 
$u$ is assigned interval $\left[\Pi_j^{-1}(u),n+1\right]$ and $v$ is assigned
interval $\left[0,\Pi_j^{-1}(v)\right]$. Since $0<\Pi_j^{-1}(u)<
\Pi_j^{-1}(v)<n+1$, it follows that $u$ is adjacent to $v$ in $I_{hj}$,
$\forall j\in[k]$.
\paragraph{Case 4: $(h=i)$} 
If $h=p$ and $j=k$, then $i'=p+1$ and therefore $u$ is
assigned $\left[\Pi_k^{-1}(u),n+1\right]$ and $v$ is assigned
$\left[\Pi_k^{-1}(v),\Pi_k^{-1}(v)\right]$. If not, then $u$ is assigned
the point interval $\left[\Pi_j^{-1}(u),\Pi_j^{-1}(u)\right]$ and $v$
is assigned $\left[0,\Pi_j^{-1}(v)\right]$. In either case, since
$\Pi_j^{-1}(u)< \Pi_j^{-1}(v)$, the two vertices are adjacent.
\paragraph{Case 5: $(h=i')$} Since $h\le p=\chi(G_\poset)-1$,
it implies that $i,i'\le p$. Therefore, if $h=p$ and $j=k$, then
$u$ and $v$ are assigned $\left[\Pi_j^{-1}(u),n+1\right]$ and
$\left[\Pi_j^{-1}(v),n+1\right]$ respectively. If not, then $v$ is
assigned the point interval $\left[\Pi_j^{-1}(v),\Pi_j^{-1}(v)\right]$
and $u$ is assigned $\left[\Pi_j^{-1}(u),n+1\right]$. Again, since
$\Pi_j^{-1}(u)< \Pi_j^{-1}(v)$, in either case the two vertices are
adjacent. Hence proved.
\bqed
\end{proof}
\begin{clm}
If $u$ and $v$ are not adjacent in $G_\poset$, then there exists some
$I\in\igr$ such that $\{u,v\}\notin E(I)$.
\end{clm}
\begin{proof}
Again let $u$ be colored $i$ and $v$ be colored $i'$. Recall that
$k\ge2$. If $i=i'$, then by construction it is clear that $u$ and $v$ are
not adjacent in $I_{i1}$ if $i\ne p+1$ and when $i=p+1$, then they are not
adjacent in $I_{pk}$. Therefore, without loss of generality we will assume
that $i<i'$. Since $u$ and $v$ are not adjacent in $G_\poset$, they are
incomparable in $P$ and therefore, there exists some $l\in[k]$ such that
$u>v$ in $L_l$ which in turn implies that $\Pi_l^{-1}(u)>\Pi_l^{-1}(v)$.
There are 2 possible cases:
\paragraph{Case 1: $(i<p)$}
Since $i<i'$, in $I_{il}$, $u$ and $v$ are assigned
intervals $\left[\Pi_l^{-1}(u),\Pi_l^{-1}(u)\right]$ and
$\left[0,\Pi_l^{-1}(v)\right]$ respectively and therefore, since
$\Pi_l^{-1}(u)>\Pi_l^{-1}(v)$ $u$ and $v$ are not adjacent in $I_{il}$.
\paragraph{Case 2: $(i=p)$}
Clearly $i'=p+1$. If $l<k$, then it is similar to the
previous case. If $l=k$, then, in $I_{pk}$, $u$ and
$v$ are assigned $\left[\Pi_k^{-1}(u),n+1\right]$ and
$\left[\Pi_k^{-1}(v),\Pi_k^{-1}(v)\right]$ respectively. Since
$\Pi_l^{-1}(u)>\Pi_l^{-1}(v)$, $u$ and $v$ are not adjacent in $I_{pk}$.
\bqed
\end{proof}
Hence we have proved Theorem \ref{thm:lowerBoundPosetdim}. 

Consider a complete $k$-partite graph $G$ on $n=qk$ vertices where $q,k>1$,
i.e. $V(G)=V_1\uplus V_2\uplus\cdots\uplus V_k$ is a partition of $V(G)$
where $|V_i|=q$. For any two vertices $x\in V_i$ and $y\in V_j$, $\{x,y\}\in
E(G)$ if and only if $i\ne j$. $G$ is a comparability graph and here is
one transitive orientation of $G$: for every pair of adjacent vertices
$u\in V_i$ and $v\in V_j$, where $u,v\in[k]$ and $i\ne j$, make $u<v$
if and only if $i<j$. Let $\poset$ be the resulting poset. It is
an easy exercise to show that $\dim(\poset)=2$.  The chromatic number
of $G$ is $k$ and Roberts \cite{recentProgressesInCombRoberts} showed
that its boxicity is $k$. From Theorem \ref{thm:lowerBoundPosetdim}
it follows that $\dim(\poset)\ge\frac{k}{k-1}$. Therefore, the
complete $k$-partite graph serves as a tight example for Theorem
\ref{thm:lowerBoundPosetdim}. 

However, it would be interesting to see if
there are posets of higher dimension for which Theorem
\ref{thm:lowerBoundPosetdim} is tight.

\section{Boxicity of the extended double cover}
In this section, we will prove Lemma \ref{lem:boxCovGraph}. But first, we
will need some definitions and lemmas.
\begin{definition}
Let $H$ be an $AB$ bipartite graph. The associated co-bipartite graph
of $H$, denoted by $H^*$ is the graph obtained by making the sets $A$
and $B$ cliques, but keeping the set of edges between vertices of $A$
and $B$ identical to that of $H$, i.e. $\forall u\in A, v\in B$, $\{u,v\}\in
E(H^*)$ if and only if $\{u,v\}\in E(H)$. 
\end{definition}
The associated co-bipartite graph $H^*$ is not to be confused with the 
complement of $H$ (i.e. $\overline{H}$) which is also a co-bipartite graph.

\begin{definition}{(Canonical interval representation of a co-bipartite
interval graph:)} \label{def:canonCobipartite}
Let $I$ be an $AB$ co-bipartite interval graph. A canonical
interval representation of $I$ satisfies: $\forall u\in A$, $l(u)=l$
and $\forall u\in B$ $r(u)=r$, where the points $l$ and $r$ are the
leftmost and rightmost points respectively of the interval representation.
\end{definition}
We claim that such a representation exists for every $AB$ co-bipartite
interval graph. Note that if $I$ is a complete graph, the claim is
trivially true. Therefore we take $I$ to be non-complete.  Consider any
interval representation of $I$. Since $A$ is a clique there exists a
point, say $l$ which is contained in all intervals corresponding to
vertices in $A$. Similarly, let $r$ be a point in the intersection of
intervals corresponding to vertices of $B$. Since $I$ is non-complete,
it is clear that $l\ne r$. By definition of $l$ and $r$ we have
$l(u)\le l\le r(u),\ \ \forall u\in A$ and $l(u)\le r\le r(u)\ \
\forall u\in B$. Without loss of generality we can assume that $l<r$
and as a result $r(u)\ge l$ and $l(u)\le r$ for all vertices $u$. This
means no interval ends before the point $l$ and no interval starts after
the point $r$. Hence, it follows that for any interval containing $l$,
we can make $l$ its left end point and for an interval containing $r$,
we can make $r$ its right end point. Therefore, we have a canonical
interval representation of $I$.

The following lemma is easy to verify. 
\begin{lemma}\label{lem:intInt}
Consider two closed intervals on the real line with left end points $l_1$,
$l_2$ and right end points $r_1$, $r_2$. Then, the two intervals
intersect if and only if $l_1\le r_2$ and $l_2\le r_1$. In other words,
the two intervals do not intersect if and only if $r_1<l_2$ or $r_2<l_1$.
\end{lemma}
\begin{lemma}\label{lem:bipcobip}
Let $H$ be an $AB$ bipartite graph and $H^*$ its associated co-bipartite
graph. If $H^*$ is a non-interval graph, then
\[
\frac{\boxi(H^*)}{2}\le \boxi(H)\le \boxi(H^*).
\]
If $H^*$ is an interval graph, then $\boxi(H)\le2$.
\end{lemma}
\begin{proof}
We first show that $\boxi(H)\le \boxi(H^*)$. Let $\boxi(H^*)=k\ge2$
and $H^*=I_1\cap I_2\cap\ldots\cap I_k$, where $I_i$ are interval
graphs.  Note that since $I_i$ is a supergraph of a co-bipartite graph,
it is a co-bipartite interval graph.  Let us consider a canonical
interval representation for each $I_i$ and further assume that the right
end points of all vertices in $A$ and left end points of all vertices in
$B$ are distinct. Let $I_1'$ be the interval graph obtained by making
$r(u,I_1')=l(u,I_1')=l(u,I_1)\ \ \forall u\in B$ and keeping the rest
of the intervals unchanged. Similarly, let $I_2'$ be the interval graph
obtained by making $l(u,I_2')=r(u,I_2')=r(u,I_2)\ \ \forall u\in A$. Due
to our assumption of distinct end points it is clear that $A$ and $B$
are independent sets in $I_1'$ and $I_2'$ respectively. Suppose $u\in A$
and $v\in B$. For $i\in[2]$:
\begin{eqnarray*}
\{u,v\}\in E(I_i') &&\Longleftrightarrow r(u,I_i')\ge l(v,I_i')\\
\textrm{(by construction of $I'$ from $I$)} &&\Longleftrightarrow
r(u,I_i)\ge l(v,I_i)\\ 
&&\Longleftrightarrow \{u,v\}\in E(I_i)
\end{eqnarray*}
From this, we immediately see that $H=I_1'\cap I_2'\cap I_3\cap\ldots\cap
I_k$. 

Now suppose $\boxi(H^*)=1$, i.e. $H^*$ is an interval graph. Then we set
$I_1=I_2=H^*$ and proceed as in the previous case. Hence, $\boxi(H)\le
2$. Note that this inequality is tight: take for example $H=C_4$,
the cycle of length 4. $H^*$ is $K_4$ and therefore an interval graph,
but $C_4$ is not.

Now we show that $\boxi(H^*)\le2\boxi(H)$. Let $\boxi(H)=l$ and $H=I_1\cap
I_2\cap\ldots\cap I_l$, where $I_i$ are interval graphs. For each $I_i$, we
create two interval graphs $I_{2i-1}'$ and $I_{2i}'$ as follows: Consider an
interval representation of $I_i$. Let $l_i=\min_{u\in V}l(u,I_i)$ and
$r_i=\max_{u\in V}r(u,I_i)$, the leftmost and rightmost points in the
interval representation respectively. $I'_{2i-1}$ and $I'_{2i}$ are defined
as follows:
\[
\begin{array}{ll}
l(u,I'_{2i-1}) = l_i \textrm{ and } r(u,I'_{2i-1}) = r(u,I_i), & \forall u\in A,\\
r(u,I'_{2i-1}) = r_i \textrm{ and } l(u,I'_{2i-1}) = l(u,I_i), & \forall u\in B,\\
l(u,I'_{2i}) = l_i \textrm{ and } r(u,I'_{2i}) = r(u,I_i), & \forall u\in B,\\
r(u,I'_{2i}) = r_i \textrm{ and } l(u,I'_{2i}) = l(u,I_i), & \forall u\in A.
\end{array}
\]
Now we show that $H^*=\bigcap_{i=1}^{2l}I_i'$. From the definitions it
is clear that in each $I_i'$, $A$ and $B$ are cliques-- for example,
the interval corresponding to every vertex in $A$ in $I'_{2i-1}$ contains
$l_i$. Therefore we will assume that $u\in A$ and $v\in B$.
\[
\begin{array}{ccl}
\{u,v\}\in E(H^*) &\Longrightarrow& \{u,v\}\in E(H) \\
&\Longrightarrow& \{u,v\}\in E(I_i),\ \ \forall i=1,2,\ldots,l \\
\textrm{(From Lemma \ref{lem:intInt})}&\Longrightarrow& l(u,I_i)\le
r(v,I_i) \textrm{ and } l(v,I_i)\le r(u,I_i) \\
\end{array}
\]
In $I'_{2i-1}$, $l(u,I'_{2i-1})=l_i\le r_i\le r(v,I'_{2i-1})$ and
$l(v,I'_{2i-1})=l(v,I_i)\le r(u,I_i)=r(u,I_{2i-1}')$ and in $I'_{2i}$,
$l(v,I'_{2i})=l_i\le r_i\le r(u,I'_{2i})$ and
$l(u,I'_{2i})=l(u,I_i)\le r(v,I_i)=r(v,I_{2i}')$.
Therefore $u$ and $v$ are adjacent in both $I'_{2i-1}$ and $I'_{2i}$. 
Now suppose
\begin{eqnarray*}
\{u,v\}\notin E(H^*)&&\Longrightarrow \{u,v\}\notin E(H)\\
&&\Longrightarrow \exists I_j \textrm{ such that } \{u,v\}\notin E(I_j)
\end{eqnarray*}
In the interval representation of $I_j$, if $r(u,I_j)<l(v,I_j)$, then,
by definition $r(u,I'_{2j-1})<l(v,I'_{2j-1})$ and hence, $\{u,v\}\notin
E(I_{2j-1}')$. If $r(v,I_j)<l(u,I_j)$, then, $r(v,I'_{2j})<l(u,I_{2j}')$
and therefore, $\{u,v\}\notin E(I_{2j}')$. Hence proved.
\qed
\end{proof}

\subsection{Proof of Lemma \ref{lem:boxCovGraph}}\label{sec:proofCover}
\paragraph{$\boxi(G_c) \le \boxi(G)+2$:} Let $\boxi(G)=k$ and $G=I_1\cap
I_2\cap\ldots\cap I_k$ where $I_i$s are interval graphs. For each $I_i$,
we construct interval graphs $I'_i$ with vertex set $V(G_c)$ as follows:
Consider an interval representation for $I_i$. For every vertex $u$ in
$I_i$, we assign the interval of $u$ to $u_A$ and $u_B$ in $I_i'$. Let
$I'_{k+1}$ and $I'_{k+2}$ be interval graphs where (1) all vertices in $A$
are adjacent to all the vertices in $B$ (2) In $I'_{k+1}$ $A$ induces
a clique and $B$ induces an independent set while in $I'_{k+2}$ it is
the other way round. Now we show that $G_c=I'_1\cap I'_2\cap\ldots\cap
I'_{k+2}$. It is very easy to see that $\{u_A,u_B\}\in E(I'_i)\ \
\forall i\in[k+2]$. Suppose $u$ and $v$ are distinct vertices in $G$.
\begin{eqnarray*}
\{u_A,v_B\}\in E(G_c) &&\Longrightarrow \{u,v\}\in E(G)\\
&&\Longrightarrow \{u,v\}\in E(I_i), i\in[k]\\
&&\Longrightarrow \{u_A,v_B\}\in E(I'_i), i\in[k].
\end{eqnarray*}
Also, by definition it is clear that $\{u_A,v_B\}$ is an edge in both
$I'_{k+1}$ and $I'_{k+2}$. Therefore, $I'_i$s are all supergraphs of $G_c$.
\begin{eqnarray*}
\{u_A,v_B\}\notin E(G_c) &&\Longrightarrow \{u,v\}\notin E(G)\\
&&\Longrightarrow \exists I_j, j\in[k]\textrm{ such that } \{u,v\}\notin
E(I_j)\\
&&\Longrightarrow \{u_A,v_B\}\notin E(I'_j).
\end{eqnarray*}
$A$ and $B$ induce independent sets in $I'_{k+2}$ and $I'_{k+1}$
respectively. Hence, $G_c=I'_1\cap I'_2\cap\cdots\cap I'_k\cap I'_{k+1}\cap
I'_{k+2}$ and therefore $\boxi(G_c)\le\boxi(G)+2$.

\paragraph{$\boxi(G) \le2\boxi(G_c)$:} We will assume without loss of
generality that $|V(G)|>1$. This implies $G_c$ is not a complete graph
and therefore $\boxi(G_c)>0$. Let us consider the associated co-bipartite
graph of $G_c$, i.e. $G_c^*$. We will show that $\boxi(G)\le \boxi(G_c^*)$
and the required result follows from Lemma \ref{lem:bipcobip}. Let
$\boxi(G_c^*)=p$ and $G_c^*=J_1\cap J_2\cap\ldots\cap J_p$ where $J_i$s
are interval graphs. Let us assume canonical interval representation for
each $J_i$ (recall Definition \ref{def:canonCobipartite}). Corresponding
to each $J_i$, we construct an interval graph $J'_i$ with vertex set
$V(G)$ as follows: The interval for any vertex $u$ is the intersection
of the intervals of $u_A$ and $u_B$, i.e.  $l(u,J'_i)=l(u_B,J_i)$ and
$r(u,J'_i)=r(u_A,J_i)$. Note that since $u_A$ and $u_B$ are adjacent in
$J_i$, their intersection is non-empty.

Now we show that $G=\bigcap_{i=1}^pJ'_i$. First we consider two adjacent
vertices $u$ and $v$.
\begin{eqnarray*}
\{u,v\}\in E(G) &&\Longrightarrow \{u_A,v_B\},\{u_B,v_A\}\in E(G_c^*)\\
&&\Longrightarrow \{u_A,v_B\},\{u_B,v_A\}\in E(J_i), \forall i\in[p]\\
\textrm{(From Lemma \ref{lem:intInt})}&&\Longrightarrow l(v_B,J_i)\le
r(u_A,J_i) \textrm{ and } l(u_B,J_i)\le r(v_A,J_i), \forall i\in[p]\\
\textrm{(By definition of $J_i'$)}&&\Longrightarrow l(v,J'_i)\le r(u,J'_i)
\textrm{ and } l(u,J'_i)\le r(v,J'_i), \forall i\in[p]\\ \textrm{(From
Lemma \ref{lem:intInt})}&&\Longrightarrow \{u,v\}\in E(J'_i), \forall
i\in[p]
\end{eqnarray*}
Therefore, each $J'_i$ is a supergraph of $G$. Now, suppose $u$ and $v$ are
not adjacent.
\begin{eqnarray*}
\{u,v\}\notin E(G) &&\Longrightarrow \{u_A,v_B\}\notin E(G_c^*)\\
&&\Longrightarrow \exists J_j \textrm{ such that } \{u_A,v_B\}\notin
E(J_j)\\
\textrm{(From Lemma \ref{lem:intInt})}&&\Longrightarrow r(u_A,J_j) <
l(v_B,J_j) \textrm{ or } r(u_B,J_j) <l(v_A,J_j)\\
(\textrm{Since $J_j$ has a canonical interval representation})&&
\Longrightarrow r(u_A,J_j) < l(v_B,J_j)\\
\textrm{(By definition of $J_j'$)}&&\Longrightarrow r(u,J'_j) < l(v,J'_j)\\
\textrm{(From Lemma \ref{lem:intInt})}&&\Longrightarrow \{u,v\}\notin E(J'_j)
\end{eqnarray*}
Hence, $G=J'_1\cap J'_2\cap\cdots\cap J'_p$ and from Lemma
\ref{lem:bipcobip} we have $\boxi(G)\le\boxi(G_c^*)\le2\boxi(G_c)$.

\bibliography{docsdb}
\end{document}